\documentclass{amsart}

\usepackage[utf8]{inputenc}
\usepackage{amssymb}
\usepackage[dvipsnames]{xcolor}
\usepackage{tikz}
\usetikzlibrary{calc,math}

\usepackage[colorlinks=false]{hyperref}
\usepackage[msc-links]{amsrefs}
\usepackage[capitalise]{cleveref}

\newtheorem{theorem}{Theorem}
\newtheorem{lemma}[theorem]{Lemma}
\newtheorem{corollary}[theorem]{Corollary}
\newtheorem{proposition}[theorem]{Proposition}

\theoremstyle{definition}
\newtheorem{definition}[theorem]{Definition}
\newtheorem*{notation}{Notation}
\newtheorem{observation}[theorem]{Observation}

\theoremstyle{remark}
\newtheorem{remark}[theorem]{Remark}

\title{Finite rigid sets in sphere complexes}
\author{Edgar A. Bering IV}
\address{Department of Mathematics \& Statistics, San José State University, San José, California}
\email{edgar.bering@sjsu.edu}
\author{Christopher J. Leininger}
\address{Department of Mathematics, Rice University, Houston, Texas}
\email{cjl12@rice.edu}
\date{April 4, 2022}

\DeclareMathOperator{\Aut}{\mathrm{Aut}}
\DeclareMathOperator{\Mod}{\mathrm{Mod}}
\DeclareMathOperator{\Out}{\mathrm{Out}}

\DeclareMathOperator{\PGL}{\mathrm{PGL}}

\newcommand{\Sc}{\mathcal{S}} 
\definecolor{red}{RGB}{213,94,0}
\definecolor{yellow}{RGB}{240,228,66}
\definecolor{orange}{RGB}{230,159,0}
\definecolor{blue}{RGB}{0,114,178}
\definecolor{magenta}{RGB}{204,121,167}
\definecolor{cyan}{RGB}{86,180,233}
\definecolor{green}{RGB}{0,158,115}

\begin{document}

\begin{abstract}
A subcomplex $X\leq \mathcal{C}$ of a simplicial complex is strongly rigid if every locally injective, simplicial map $X\to\mathcal{C}$ is the restriction of a unique automorphism of $\mathcal{C}$. Aramayona and the second author proved that the curve complex of an orientable surface can be exhausted by finite strongly rigid sets. The Hatcher sphere complex is an analog of the curve complex for isotopy classes of essential spheres in a connect sum of $n$ copies of $S^1\times S^2$. We show that there is an exhaustion of the sphere complex by finite strongly rigid sets for all $n\ge 3$ and that when $n=2$ the sphere complex does not have finite rigid sets.
\end{abstract}

\maketitle

\section{Introduction}

The curve complex $\mathcal{C}(S)$ of a surface $S$ is the flag simplicial complex whose $k$-simplicies are sets of $k+1$ isotopy classes of pairwise disjoint essential non-peripheral simple closed curves. The curve complex has been found to exhibit remarkable rigidity properties~\cites{ivanov, korkmaz, luo, shackleton, irmak-superinjective}; with  vast generalizations~\cites{brendle-margalit,McLeay,disarlo-koberda-Gonzalez}. Among these rigidity properties, Aramayona and the second author proved that for each orientable surface of finite type $S$ there is a finite subcomplex $X \subseteq \mathcal{C}(S)$ such that every locally injective, simplicial map $X \to \mathcal{C}(S)$ is the restriction of a unique element of $\Aut(\mathcal{C}(S))$. Such a subcomplex is called \emph{strongly rigid}~\cite{aramayona-leininger}. In subsequent work, they also constructed an exhaustion of the curve complex $\mathcal{C}(S) = \cup_n X_n$ by a nested sequence of finite rigid subcomplexes $X_n$~\cite{aramayona-leininger-2}. Subsequently, a plethora of surface complexes have been shown to have (exhaustions by) finite rigid sets~\cites{ik-nonorientable,irmak-nonorientable,irmak-nonorientable-2,Mate,Mate2,hhlm-pants,shinkle-arc,shinkle-flip}.

The curve complex is used to study the mapping class group of a surface. In the long running analogy between mapping class groups of surfaces and outer automorphisms of free groups, several complexes have been introduced to play an analogous role. The most topological of these is the \emph{sphere complex} $\Sc(M_{n,0})$ of $M_{n,0} = \#_n S^1\times S^2$, the connect sum of $n$ copies of $S^1\times S^2$: a simplicial complex whose $k$-simplicies are sets of $k+1$ isotopy classes of pairwise disjoint essential non-peripheral embedded 2-spheres in $M_{n,0}$~\cite{hatcher}. The mapping class group $\Mod(M_{n,0})$ surjects $\Out(F_n)$~\cite{laudenbach}*{Théorème III} with finite kernel. Like the curve complex, the sphere complex is \emph{simplicially rigid}: Aramayona and Souto prove that for $n\ge 3$ the group $\Aut(\Sc(M_{n,0}))$ is isomorphic to $\Out(F_n)$~\cite{aramayona-souto}. Motivated by the analogy between the curve complex and the sphere complex, we establish the existence of finite strongly rigid sets in the sphere complex of a connect sum of copies of $S^1\times S^2$.

\begin{theorem}\label{rigid set theorem}
For $n \geq 3$, there exists a finite strongly rigid simplicial complex $X\subseteq \Sc(M_{n,0})$; that is,
    for any locally injective, simplicial map $f \colon X\to \Sc(M_{n,0})$
    there exists a unique element $\phi \in \Aut(\Sc(M_{n,0}))$ such that
    \[ f = \phi|_X. \]
\end{theorem}

We also show that the sphere complex can be exhausted by finite strongly rigid sets.

\begin{theorem}\label{rigid exhaustion theorem}
For $n\geq 3$, there exists a nested family of finite strongly rigid simplicial complexes $X_j\subseteq \Sc(M_{n,0})$ such that
\[ \Sc(M_{n,0}) = \cup_j X_j. \]
\end{theorem}

Our proofs of \cref{rigid set theorem,rigid exhaustion theorem} start by constructing an exhaustion of $\Sc(M_{n,0})$ by sets that are \emph{geometrically rigid}: every locally injective simplical map is the restriction of the action of some element of $\Mod(M_{n,0})$,
and any two such elements induce the same element of $\Out(F_n)$. In \cref{strong rigidity section} we combine this uniqueness with the exhaustion to produce a new proof of Aramayona and Souto's isomorphism $\Out(F_n) \cong \Aut(\Sc(M_{n,0}))$. We conclude strong rigidity by using this isomorphism.

The restriction $n\geq 3$ is necessary. While $\Sc(M_{2,0})$ is related to the Farey graph, which has finite rigid sets, each separating sphere in $M_{2,0}$ determines a unique triangle of $\Sc(M_{2,0})$ attached to the natural Farey-subgraph (often called a \emph{fin}). We show that these fins are an obstruction to the existence of rigid sets and we use them to prove a negative result.

\begin{theorem}\label{no rank two rigid}
The sphere complex $\Sc(M_{2,0})$ does not contain any finite rigid subcomplex.
\end{theorem}

\begin{remark}
To clarify, a {\em rigid} subcomplex of a simplicial complex is one for which any locally injective simplicial map back into the complex extends to an automorphism. If that automorphism is unique, then the subcomplex is {\em strongly rigid}. If the complex in question is a sphere complex of a manifold and the automorphism is realized by a homeomorphism, then the subcomplex is {\em geometrically rigid}.
\end{remark}

\section*{Acknowledgements}

The authors thank Valentina Disarlo for asking about the finite rigidity of $\Out(F_n)$ complexes and suggesting we determine strong rigidity.  EB also thanks Shaked Bader for a helpful conversation, Rice University for hospitality when this work was begun, and was supported by the Azreili foundation. CL was partially supported from NSF grant DMS-2106419.

\section{Building blocks}

The sphere complex is a flag complex, so we will focus on the 1-skeleton. In keeping with tradition in the area, unless specified when we say \emph{sphere $S$ in $M$} we mean isotopy class of \emph{essential} ($S$ does not bound a 3-ball), \emph{non-peripheral} ($S$ is not isotopic into $\partial M$) smoothly embedded sphere in $M$, or a representative of the isotopy class, with the context indicating which is intended.

\begin{notation}
    Let $M_{n,s}$ denote the connect sum of $S^3$ with $n$ copies of $S^1\times S^2$ and with $s$ disjoint open balls removed.
\end{notation}

While this article is focused on $\Sc(M_{n,0})$, along the way we will also need to understand certain other $\Sc(M_{n,s})$. In both cases, the elementary building blocks are parallels of those in the surface setting.

\begin{definition}
    Any manifold homeomorphic to $M_{0,3}$ is called a \emph{pair of pants}.
\end{definition}

\begin{definition}
    A maximal collection of disjoint spheres $P \subseteq \Sc(M_{n,s})^{(0)}$ is called a \emph{pants decomposition}.
    Each connected component of $M_{n,s}\setminus P$ is the interior of a pair of pants.
\end{definition}

The following definition and lemma parallel the surface case.

\begin{definition}[\cite{aramayona-leininger}*{Definition 2.2}]
    Let $X \subset \Sc(M_{n,s})$ be a subcomplex. Two spheres $\alpha, \beta \in X^{(0)}$ that intersect essentially have \emph{$X$-detectable intersection} if there are two pants decompositions $P_\alpha, P_\beta \subset X^{(0)}$ such that
    \[ \alpha \in P_\alpha, \beta \in P_\beta, \text{ and } P_\alpha - \alpha = P_\beta - \beta.\]
    In this case $N(\alpha \cup \beta) = M_{0,4}$.
\end{definition}

\begin{lemma}[\cite{aramayona-leininger}*{Lemma 2.3}]\label{detect-to-detect}
    Let $X \subset \Sc(M_{n,s})$ be a subcomplex and suppose that $f \colon X\to \Sc(M_{n,s})$ is a locally injective simplicial map. If $a, b \in X^{(0)}$ have $X$-detectable intersection
    then $f(a), f(b)$ have $f(X)$-detectable intersection, and hence 
    a closed regular neighborhood $\mathcal{N}(f(a) \cup f(b)) \cong M_{0,4}$.
\end{lemma}

\begin{proof}
    Let $P_a,P_b \subset X^{(0)}$ be the pants decompositions detecting the intersection of $a$ and $b$. Since $f$ is locally injective and simplicial, $f(P_a)$ and $f(P_b)$ are again pants decompositions. Moreover, $f$ preserves membership and set equality, so $f(P_a)$ and $f(P_b)$ detect the intersection of $f(a)$ and $f(b)$.
\end{proof}

In order to build strongly rigid sets we capitalize on subcomplexes of $\Sc(M_{n,0})$ coming from {\em punctured 3-spheres}, $M_{0,s}$. The sphere complex of a punctured 3-sphere $\Sc(M_{0,s})$ is finite and admits a purely combinatorial description, due to the following.

\begin{lemma}\label{partition-determines}
A sphere $S$ in $M_{0,s}$ is determined by the partition of $\partial M_{0,s}$ induced by the connected
components of $M_{0,s}\setminus S$.
\end{lemma}
\begin{proof}
    Suppose $a,b$ are two spheres in $M$ that induce the same partition of $\partial M_{0,s}$ into two sets.  Both of the sets are non-empty and contain at least two components, since by our definition spheres are essential.
The closures of each of the two connected components of $M_{0,s}\setminus a$ are homeomorphic to a corresponding connected component of $M_{0,s}\setminus b$ via a homeomorphism that is identity on $\partial M_{0,s} $ and takes the $a$ boundary components to the $b$ boundary components. After possibly composing with a Dehn twist in $b$ these homeomorphisms can be extended to a homeomorphism $h: M_{0,s} \to M_{0,s}$ such that $h(a) = b$ and $h$ is the identity on $\partial M_{0,s}$. Since $H_2(M_{0,s})$ is generated by the classes of the boundary spheres, $h_\ast \colon H_2(M_{0,s}) \to H_2(M_{0,s})$ is the identity. Moreover, $M_{0,s}$ is simply connected, so the Hurewicz homomorphism is an isomorphism; by the naturality of the Hurewicz isomorphism we conclude $h_\sharp \colon \pi_2(M_{0,s}) \to \pi_2(M_{0,s})$ is the identity. Thus $a$ is homotopic to $b$, and by Laudenbach's theorem~\cite{laudenbach}*{Théorème I} we conclude that $a$ is in fact isotopic to $b$.
\end{proof}

\begin{lemma}\label{punctured-homeo}
    The natural map $\Mod(M_{0,s}) \to \Aut(\Sc(M_{0,s}))$ is
    a surjection.
\end{lemma}

\begin{proof}
    If $s \le 3$, the complex $\Sc(M_{0,s})$ is empty or a singleton and the result is trivial. The complex $\Sc(M_{0,4})$ is a disconnected 3-vertex graph and it is easy to construct homeomorphisms of $M_{0,4}$ that generate  $\Aut(\Sc(M_{0,4}))$. So we suppose $s\geq 5$.
    
    Let $[s]$ denote the set of connected components of $\partial M_{0,s}$,
    thought of as in bijection with the set $\{1, \ldots, s\}$. By \cref{partition-determines}, each vertex of $\Sc(M_{0,s})$ corresponds to a two-piece partition of $[s]$ where each piece has size at least 2.
    For brevity we refer to the cardinality of the smaller partition piece determined by a sphere $S$ as the \emph{size} of $S$. The vertices of $\Sc(M_{0,s})$ are partitioned by size. Each sphere of size greater than 2 is determined uniquely by the set of size 2 spheres disjoint from it, hence $\Aut(\Sc(M_{0,s})) \simeq \Aut(\Sc_2(M_{0,s}))$, where $\Sc_2(M_{0,s})$ is the induced subcomplex spanned by size 2 spheres.
    
    The 1-skeleton $\Sc_2(M_{0,s})^{(1)}$ is the graph whose vertices are $2$-element subsets of $[s]$ and edges are between disjoint sets---this is known in the literature as the Kneser graph $K(s, 2)$. In turn, $K(s,2)$ is the complement of the line graph $L(K_s)$, where $K_s$ is the complete graph on vertex set $[s]$. Since $s\ge 5$, the Whitney isomorphism theorem implies that $\Aut(\Sc_2(M_{0,s})) \simeq \Aut(K_s)$~\cite{whitney-theorem-ref}. The group $\Aut(K_s) \simeq \Aut([s])$, the symmetric group on $[s]$, thus $\Aut(\Sc_2(M_{0,s}))\simeq \Aut([s])$. Moreover this isomorphism is given by the action of $\Aut([s])$ on 2-element subsets of $[s]$, and so is compatible with the natural actions of $\Mod(M_{0,s})$ on $\Sc_2(M_{0,s})$ and $[s]$. The natural action of $\Mod(M_{0,s})$ on $\partial M_{0,s}$ clearly surjects $\Aut([s])$, this completes the proof.
\end{proof}

An easy consequence is the following.

\begin{corollary} \label{C:induced homeo}
If $N \cong M_{0,n} \cong N'$ and $\phi \colon \Sc(N) \to \Sc(N')$ is any simplicial isomorphism, then there exists a homeomorphism $h \colon N \to N'$ so that $h(z) = f(z)$ for every sphere $z \in \Sc(N)$. \qed
\end{corollary}

The following observation relies on \cref{partition-determines} and is a useful combinatorial criterion for distinguishing among spheres in a pair of nicely arranged $M_{0,4}$-submanifolds of some $M_{n,s}$.

\begin{lemma}\label{six-holed-evil-twins-cross}
Suppose $a, b, c \subset M_{n,s}$ are spheres such that $b\cap c =\emptyset$ and $b$ and $c$ both essentially intersect $a$ exactly once, and every component of $\partial \mathcal{N}(a\cup b\cup c)$ is either essential or peripheral. Let $b'\subseteq \mathcal{N}(a\cup b)$ be the sphere not isotopic to $a$ or $b$ and similarly define $c' \subseteq \mathcal{N}(a\cup c)$. Then $b'$ and $c'$ intersect,
and both intersect $b$ and $c$. (Note that $b'$ and $c'$ are unique by \cref{partition-determines}.)
\end{lemma}
\begin{proof}
Observe that $\mathcal{N}(a\cup b\cup c)\cong M_{0,6}$, and since each boundary sphere is essential or peripheral in $M_{n,s}$, the inclusion $\mathcal{N}(a\cup b\cup c) \to M_{n,s}$ preserves intersection. Fix an identification of $\partial M_{0,6}$ with $[6] = \{1,\ldots 6\}$. By \cref{partition-determines} the spheres are determined by the partition of $[6]$, and we identify each sphere with one of its partition pieces. Up to a choice of boundary identification we have $a = \{1, 2, 3\}$, $b = \{1, 6\}$ and $c = \{3, 4\}$. The boundary $\partial N(a\cup b)$ consists of the spheres $\{1\}$, $\{6\}$, $\{2,3\}$, and $\{4, 5\}$. By again appealing to \cref{partition-determines} applied to $N(a\cup b)$ we deduce $b' = \{1,4,5\}$ and similarly $c' = \{3,5,6\}$, see \cref{evil-twin-figure}. The lemma follows since spheres in $M_{0,s}$ intersect if and only if the partitions they define are not nested.
\end{proof}

\begin{figure}[ht]
\begin{center}
\begin{tikzpicture}[every node/.style={font=\scriptsize},thick,scale=2]
\node[circle,draw=black,fill=black!10,thick] (v1) at (0:1) {$1$};
\node[circle,draw=black,fill=black!10,thick] (v2) at (60:1) {$2$};
\node[circle,draw=black,fill=black!10,thick] (v3) at (120:1) {$3$};
\node[circle,draw=black,fill=black!10,thick] (v4) at (180:1) {$4$};
\node[circle,draw=black,fill=black!10,thick] (v5) at (240:1) {$5$};
\node[circle,draw=black,fill=black!10,thick] (v6) at (300:1) {$6$};

\draw[red] ($ (v1)+(240:0.4) $) arc (240:390:0.4) to[out=120,in=300] 
                    ($ (v2)+(30:0.4) $) arc  (30:90:0.4) node[midway,above right] {$a$} to[out=180,in=0]
                    ($ (v3)+(90:0.4) $) arc (90:240:0.4) to[out=330,in=150] cycle;
\draw[yellow] ($(v6) + (150:0.2)$) arc (150:330:0.2) to[out=60,in=240] node[midway,below right] {$b$}
            ($(v1) + (330:0.2)$) arc (-30:150:0.2) to[out=240,in=60] cycle;
\draw[blue] ($(v3) + (330:0.2)$) arc (-30:150:0.2) to[out=240,in=60] node[midway,above left] {$c$}
            ($(v4) + (150:0.2)$) arc (150:330:0.2) to[out=60,in=240] cycle;
            
\draw[orange] ($(v4) + (90:0.25)$) arc (90:210:0.25) to[out=300,in=120]
              node[midway,below left] {$b'$}
              ($(v5) + (210:0.25)$) arc (210:300:0.25) to[out=30,in=210]
              ($(v1) + (300:0.25)$) arc (300:450:0.25) to[out=180,in=0] cycle;
              
\draw[magenta] ($(v5) + (180:0.3)$) arc (180:270:0.3) to[out=0,in=180]
              node[midway,below] {$c'$}
              ($(v6) + (270:0.3)$) arc (270:390:0.3) to[out=120,in=300]
              ($(v3) + (30:0.3)$) arc (30:180:0.3) to[out=270,in=90] cycle;
\end{tikzpicture} 
\caption{The configuration of intersecting spheres in $M_{0,6}$ considered in the proof of \cref{six-holed-evil-twins-cross}.\label{evil-twin-figure}}
\end{center}
\end{figure}
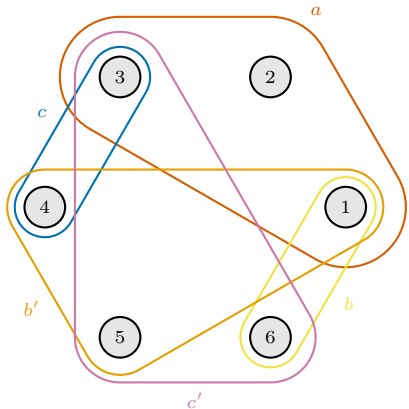

\section{The finite strongly rigid set}\label{rigid set section}

In this section we construct a particular set of spheres that we will later prove is strongly rigid for $\Sc(M_{n,0})$ when $n\geq 3$ (which we continue to assume to be the case until \Cref{S:no rigid rank 2}). In this section we will show that every locally injective simplicial map from our set comes from the action of some $h\in \Mod(M_{n,0})$.

Let $Y$ be a maximal collection of disjoint spheres whose union is non-separating, and $\mathcal{N}^\circ(Y)$ an open regular neighborhood so that
\[ M_{n,0} \setminus \mathcal{N}^\circ (Y) \cong M_{0,2n} \]
and let $Z \subset \Sc(M_{n,0})$ be the set of all spheres in $N = M_{n,0} \setminus \mathcal{N}^\circ(Y)$, so that $Z = \Sc(N)$.  There is a labeling of the components of $\partial N$ by elements of $Y$ so that $S \subset \partial N$ is labeled by $A \in Y$ if $S$ is isotopic to $Y$ via the inclusion $N \subset M_{n,0}$ (thus, each element of $Y$ appears as exactly two labels).  Given a component $S \subset \partial N$, let $\delta(S) \in Y$ denote its label.

Let $X_0$ be the subcomplex induced by $Y\cup Z$. Note that $X_0$ decomposes as the join of subcomplexes induced by $Y$ and $Z$.  In particular, for any permutation of $Y$, there exists an automorphism of $X_0$ which is the identity on $Z$ and effects the given permutation on $Y$.  Consequently, $X_0$ is not rigid.

Suppose $f \colon X_0 \to \Sc(M_{n,0})$ is a locally injective, simplicial map. Each intersecting pair of spheres in $Z$ is $X_0$-detectable, so as a consequence of \cref{detect-to-detect} the submanifold of $M_{n,0}$ filled by $f(Z)$ is connected. Since $f$ is locally injective and simplicial, $f(X_0)$ decomposes as a join of $f(Y)$ and $f(X_0)$. Thus, $f(Y)$ is a set of $n$ non-separating spheres,
$N' = M_{n,0} \setminus \mathcal{N}^\circ(f(Y)) \cong M_{0,2n}$, and $\Sc(N') = f(Z)$.  As with $N$, there is a labeling of the components of $\partial N'$ by elements of $f(Y)$, and we similarly write $\delta(S) \in f(Y)$ for the label on the component $S \subset \partial N'$.

By \cref{C:induced homeo} there is a homeomorphism $h \colon N \to N'$ such that for each $z \in Z$, $h(z) = f(z)$.  Now observe that if 
\begin{equation} \label{E:labels preserved}
\delta(h(S)) = f(\delta(S)),
\end{equation}
for every component $S \subset \partial N$, then $h \colon N \to N'$ can be extended to a homeomorphism $\hat h \colon M_{n,0} \to M_{n,0}$ which induces $f$, and we would be done.  There is no reason that \cref{E:labels preserved} should hold, however, but we will shortly show that by adding some additional spheres, it does.

Given an essential disk $D$ in $N$ with $\partial D \subset S$, where $S \subset \partial N$ is a component, we note that the regular neighborhood of $D \cup S$ is a pair of pants we denote $P(S,D)$.  The boundary of the pants, $\partial P(S,D)$, consists of $S$ together with two other spheres.

For each sphere $A \in Y$, let $A_+,A_- \subset \partial N$ be the boundary components labeled by $A$.  Consider an essential sphere $a$ in $M_{n,0}$ that intersects $A$ essentially in a single simple closed curve and is disjoint from every other sphere in $Y$.  This intersection is $X_0$--detectable and $a \cap N$ is a union of two disjoint, essential disks, $D_+,D_-$, with $\partial D_+ \subset A_+$ and $\partial D_- \subset A_-$.  Say that $a$ is {\em good for $A$} if there is a decomposition
\[ \partial P(A_+, D_+) \cup \partial P(A_-, D_-) = A_+ \cup A_- \cup S_1 \cup S_2 \cup S_3 \cup S_4, \]
into distinct spheres where $S_1,S_2$ are peripheral. After an isotopy, we may assume that $S_1,S_2 \subset \partial N$, and we do so.  We write
\[ \partial_0 (A,a) = A_+ \cup A_- \cup S_1 \cup S_2 \subset \partial N,\]
and
\[ \partial (A,a) = \partial_0(A,a) \cup S_3 \cup S_4.\]

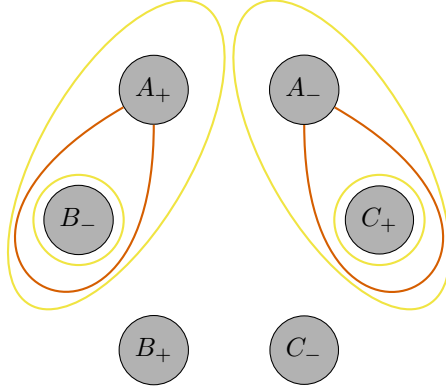
\begin{figure}[ht]
\begin{center}
\begin{tikzpicture}
\node[circle,fill=black!30,draw=black] (Am) at (60:2) {$A_-$};
\node[circle,fill=black!30,draw=black] (Ap) at (120:2) {$A_+$};
\node[circle,fill=black!30,draw=black] (Bm) at (180:2) {$B_-$};
\node[circle,fill=black!30,draw=black] (Bp) at (240:2) {$B_+$};
\node[circle,fill=black!30,draw=black] (Cm) at (300:2) {$C_-$};
\node[circle,fill=black!30,draw=black] (Cp) at (360:2) {$C_+$};

\draw[red, thick] (Ap) to[loop,out=210,in=270,min distance=4cm] (Ap);
\draw[red, thick] (Am) to[loop,out=330,in=270,min distance=4cm] (Am);

\draw[yellow, thick] (180:2) circle (0.6cm);
\draw[yellow, thick] (360:2) circle (0.6cm);

\begin{scope}[shift={(150:1.73)},rotate=60]
\draw[yellow, thick] (0,0) ellipse (2.3cm and 1cm);
\end{scope}

\begin{scope}[shift={(30:1.73)},rotate=-60]
\draw[yellow, thick] (0,0) ellipse (2.3cm and 1cm);
\end{scope}
\end{tikzpicture} \caption{A good sphere (red) for the boundary spheres labeled $A$, with the spheres $S_i$ in yellow, inside $M_{0,6}$.}\label{fig:goodsphere}
\end{center}
\end{figure}

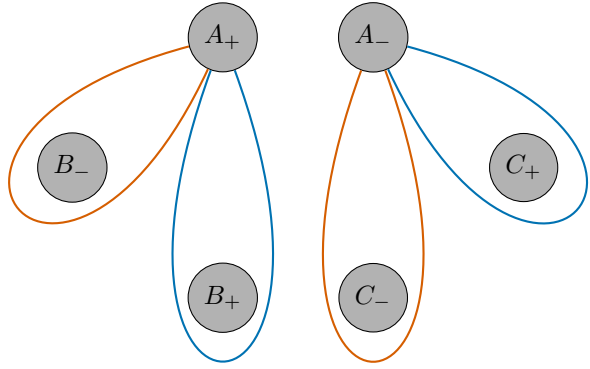
\begin{figure}[ht]
\begin{center}
\begin{tikzpicture}
\node[circle,fill=black!30,draw=black] (Am) at (60:2) {$A_-$};
\node[circle,fill=black!30,draw=black] (Ap) at (120:2) {$A_+$};
\node[circle,fill=black!30,draw=black] (Bm) at (180:3) {$B_-$};
\node[circle,fill=black!30,draw=black] (Bp) at (240:2) {$B_+$};
\node[circle,fill=black!30,draw=black] (Cm) at (300:2) {$C_-$};
\node[circle,fill=black!30,draw=black] (Cp) at (360:3) {$C_+$};

\draw[red, thick] (Ap) to[loop,out=195,in=245,min distance=4.6cm] (Ap);
\draw[blue, thick] (Ap) to[loop,out=250,in=290,min distance=5.5cm] (Ap);

\draw[blue, thick] (Am) to[loop,out=295,in=345,min distance=4.6cm] (Am);
\draw[red, thick] (Am) to[loop,out=250,in=290,min distance=5.5cm] (Am);
\end{tikzpicture} \caption{A pair of good spheres (red and blue) for the boundary sphere labeled $A$ inside $M_{0,6}$.}\label{fig:goodpair}
\end{center}
\end{figure}

For each $A \in Y$, let $a',a''$ be two disjoint spheres that are good for $A$ so that
\[ \partial (A,a') \cap \partial (A,a'') = A_+ \cup A_-,\]
and let $X$ be the union of $X_0$, together with such a good pair $a',a''$ for every $A \in Y$.  This is possible as long as $n \geq 3$, see \cref{fig:goodsphere,fig:goodpair}.

\begin{lemma} \label{L:labels preserved}
If $f \colon X \to \Sc(M_{n,0})$ is any locally injective, simplicial map and $h \colon N \to N'$ is the homeomorphism defined by the restriction $f|_{X_0}$ as above, then for every component $S \subset \partial N$, we have $\delta(h(S)) = f(\delta(S))$.
\end{lemma}
\begin{proof}
Fix $A \in Y$ and let $A_+,A_-$ be the components of $\partial N$ with $\delta(A_+) = \delta(A_-)=A$.  Let $a',a'' \in X$ be the two spheres that are good for $A$ and write
\[ \partial_0 (A,a') = A_+ \cup A_- \cup S_1' \cup S_2' \quad \mbox{ and } \quad \partial_0 (A,a'') = A_+ \cup A_- \cup S_1'' \cup S_2''.\]

We let $Z_0(A,a') \subset Z$ be the set of all spheres in $Z$ that are disjoint from $\partial (A,a')$.  Observe that $f(a')$ is a sphere that essentially intersects $f(A)$ (since the essential intersection of $A$ and $a'$ is $X$--detectable).  Therefore, we note
\begin{itemize}
    \item $f(a')$ cannot be isotoped into $N'$, and
    \item $f(a')$ is also disjoint from $f(Z_0(A,a')) = h(Z_0(A,a'))$.
\end{itemize}
From these two conditions, the only boundary components of $N'$ that $f(a')$ can essentially intersect are $h(\partial_0 (A,a'))$.  Therefore, the two components labeled $f(A)$ are necessarily contained in $h(\partial_0 (A,a'))$.  By similar reasoning, we deduce that the two components labeled $f(A)$ are also contained in $h(\partial_0(A,a''))$.  Since
\begin{eqnarray*}
h(\partial_0(A,a')) \cap h(\partial_0(A,a'')) & =& h(\partial_0(A \cap a') \cap \partial_0(A,a''))\\
&=& h(A_+ \cup A_-)\\
&= &h(A_+) \cup h(A_-),
\end{eqnarray*}
it follows that $\delta(h(A_+)) = f(A) = f(\delta(A_+))$ and $\delta(h(A_-)) = f(A) = f(\delta(A_-))$.  Since $A \in Y$ was arbitrary, this completes the proof.
\end{proof}

\begin{proposition}\label{geometrically rigid set}
The set $X$ constructed in this section is geometrically rigid.
\end{proposition}

\begin{proof}
Given a locally injective, simplicial map $f \colon X \to \Sc(M_{n,0})$, let $h \colon N \to N'$ be the homeomorphism defined by $f|_{X_0}$. By \cref{L:labels preserved}, $h$ extends to a homeomorphism $\hat h \colon M_{n,0} \to M_{n,0}$ that agrees with $f$ on $X_0$.  Now fix $A \in Y$ and consider the corresponding good spheres $a', a'' \in X$.  Let $P_A, P_{a'} \subset X$ be pants decompositions with $A \in P_A$, $a' \in P_{a'}$, and $P_0 = P_A \setminus A = P_{a'} \setminus a \subset X_0$. Since $h$ and $f$
agree on $X_0$, it follows that $\hat h$ and $f$ agree on each component of $P_A$. Hence $f(a')$ is a sphere other than $\hat h(A)$ contained in $N(\hat{h}(A) \cup \hat{h}(a'))$, the $M_{0,4}$ component of $M_{n,0}\setminus P_0$. Similarly $f(a'')$ is a sphere other than $\hat h(A)$ contained in $N(\hat{h}(A) \cup \hat{h}(a''))$. Let $e'\subset N(A\cup a')$ be the non-peripheral sphere other than $A$ and $a'$ and similarly define $e'' \subset N(A\cup a'')$. By \cref{six-holed-evil-twins-cross},
$e'$ and $e''$ intersect and both intersect $a'$ and $a''$. Thus $\hat h(e')$ and $\hat h(e'')$ intersect and both intersect $\hat h(a')$ and $\hat h(a'')$. Since $f$ is locally injective and $a'$ and $a''$ are disjoint, we conclude that $\hat h(e') \neq f(a')$ and $\hat h(e'') \neq f(a'')$. The only remaining possibility is that $\hat h(a') = f(a')$ and $\hat h(a'') = f(a'')$, as required.
\end{proof}

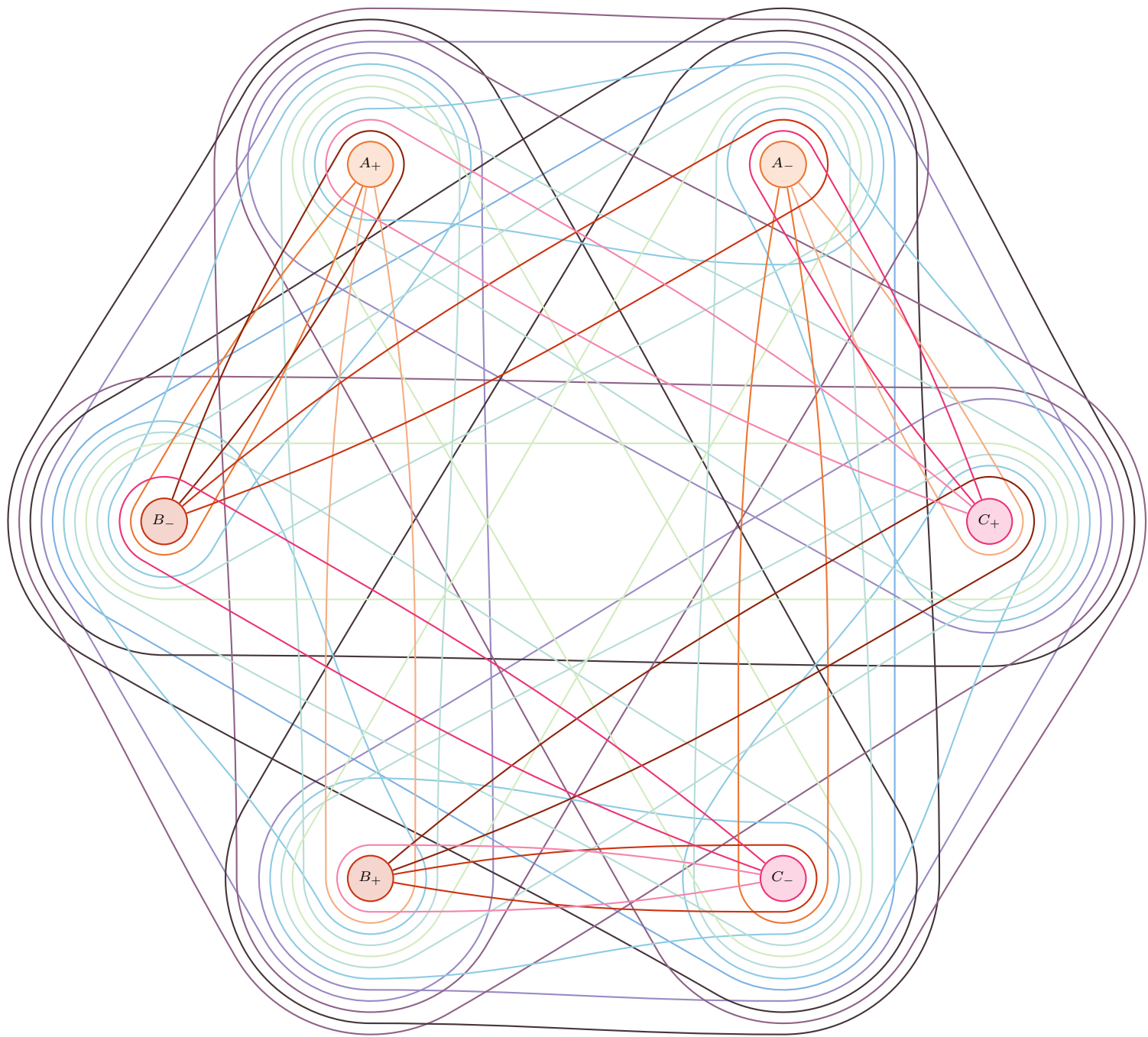
\begin{figure}[p]
\hspace*{-2.3in}
\begin{minipage}{\dimexpr\textwidth+4.6in}
\begin{center}
\tikzmath{ \dr = 7.4; \vrs = 0.2; \vrtha = 1.8; \vrtwa = 0.8; \vrga = 0.4; }

\definecolor{evensphere}{RGB}{126,178,228}
\definecolor{cyclesphere}{RGB}{155,138,196}
\definecolor{threeA}{RGB}{70,53,58}
\definecolor{threeB}{RGB}{144,99,136}

\definecolor{edgesphere}{RGB}{141,203,228}
\definecolor{chordsphere}{RGB}{181,221,216}
\definecolor{diametersphere}{RGB}{211,236,191}

\definecolor{Acolor}{RGB}{238,119,51}
\definecolor{Bcolor}{RGB}{204,51,17}
\definecolor{Ccolor}{RGB}{238,51,119}

\begin{tikzpicture}[every node/.style={font=\scriptsize},thick,scale=0.97]
\node[circle,draw=Ccolor,fill=Ccolor!20,thick] (v1) at (0:\dr) {$C_+$};
\node[circle,draw=Acolor,fill=Acolor!20,thick] (v2) at (60:\dr) {$A_-$};
\node[circle,draw=Acolor,fill=Acolor!20,thick] (v3) at (120:\dr) {$A_+$};
\node[circle,draw=Bcolor,fill=Bcolor!20,thick] (v4) at (180:\dr) {$B_-$};
\node[circle,draw=Bcolor,fill=Bcolor!20,thick] (v5) at (240:\dr) {$B_+$};
\node[circle,draw=Ccolor,fill=Ccolor!20,thick] (v6) at (300:\dr) {$C_-$};

\draw[evensphere] ($ (v2)+(0:1*\vrs + \vrtha) $) arc (0:120:1*\vrs + \vrtha) --
                    ($ (v4)+(120:1*\vrs + \vrtha) $) arc (120:240:1*\vrs + \vrtha) --
                    ($ (v6)+(240:1*\vrs + \vrtha) $) arc (240:360:1*\vrs + \vrtha) -- cycle;
\draw[cyclesphere] ($ (v1)+(240:\vrs + \vrtha) $) arc (240:390:\vrs + \vrtha) to[out=120,in=300]
                    ($ (v2)+(30:2*\vrs + \vrtha) $) arc (30:90:2*\vrs + \vrtha) to[out=180,in=0]
                    ($ (v3)+(90:2*\vrs + \vrtha) $) arc (90:240:2*\vrs + \vrtha) to[out=330,in=150] cycle;
\draw[cyclesphere] ($ (v3)+(0:\vrs + \vrtha) $) arc (0:150:\vrs + \vrtha) to[out=240,in=60]
                    ($ (v4)+(150:2*\vrs + \vrtha) $) arc (150:210:2*\vrs + \vrtha) to[out=300,in=120]
                    ($ (v5)+(210:2*\vrs + \vrtha) $) arc (210:360:2*\vrs + \vrtha) to[out=90,in=270] cycle;
\draw[cyclesphere] ($ (v5)+(120:\vrs + \vrtha) $) arc (120:270:\vrs + \vrtha) to[out=0,in=180]
                    ($ (v6)+(270:2*\vrs + \vrtha) $) arc (270:330:2*\vrs + \vrtha) to[out=60,in=240]
                    ($ (v1)+(330:2*\vrs + \vrtha) $) arc (-30:120:2*\vrs + \vrtha) to[out=210,in=30] cycle;

\draw[threeA] ($(v1) + (270:4*\vrs + \vrtha)$) arc (270:390:4*\vrs + \vrtha) to[out=120,in=300]
               ($(v2) + (30:5*\vrs + \vrtha)$) arc (30:120:5*\vrs + \vrtha) to[out=210,in=30]
               ($(v4) + (120:3*\vrs + \vrtha)$) arc (120:270:3*\vrs + \vrtha) to[out=0,in=180] cycle;
\draw[threeB] ($(v2) + (330:4*\vrs + \vrtha)$) arc (-30:90:4*\vrs + \vrtha) to[out=180,in=0]
              ($(v3) + (90:5*\vrs + \vrtha)$) arc (90:180:5*\vrs + \vrtha) to[out=270,in=90]
              ($(v5) + (180:3*\vrs + \vrtha)$) arc (180:330:3*\vrs + \vrtha) to[out=60,in=240] cycle;
\draw[threeA] ($(v3) + (30:4*\vrs + \vrtha)$) arc (30:150:4*\vrs + \vrtha) to[out=240,in=60]
              ($(v4) + (150:5*\vrs + \vrtha)$) arc (150:240:5*\vrs + \vrtha) to[out=330,in=150]
              ($(v6) + (240:3*\vrs + \vrtha)$) arc (240:390:3*\vrs + \vrtha) to[out=120,in=300] cycle;
\draw[threeB] ($(v4) + (90:4*\vrs + \vrtha)$) arc (90:210:4*\vrs + \vrtha) to[out=300,in=120]
              ($(v5) + (210:5*\vrs + \vrtha)$) arc (210:300:5*\vrs + \vrtha) to[out=30,in=210]
              ($(v1) + (300:3*\vrs + \vrtha)$) arc (300:450:3*\vrs + \vrtha) to[out=180,in=0] cycle;
\draw[threeA] ($(v5) + (150:4*\vrs + \vrtha)$) arc (150:270:4*\vrs + \vrtha) to[out=0,in=180]
              ($(v6) + (270:5*\vrs + \vrtha)$) arc (270:360:5*\vrs + \vrtha) to[out=90,in=270]
              ($(v2) + (0:3*\vrs + \vrtha)$) arc (0:150:3*\vrs + \vrtha) to[out=240,in=60] cycle;
\draw[threeB] ($(v6) + (210:4*\vrs + \vrtha)$) arc (210:330:4*\vrs + \vrtha) to[out=60,in=240]
              ($(v1) + (330:5*\vrs + \vrtha)$) arc (-30:60:5*\vrs + \vrtha) to[out=150,in=330]
              ($(v3) + (60:3*\vrs + \vrtha)$) arc (60:210:3*\vrs + \vrtha) to[out=300,in=120] cycle;

\draw[edgesphere] ($(v1) + (210:5*\vrs + \vrtwa)$) arc (210:390:5*\vrs + \vrtwa) to[out=120,in=300]
            ($(v2) + (30:\vrs + \vrtwa)$) arc (30:210:\vrs +\vrtwa) to[out=300,in=120] cycle;
\draw[edgesphere] ($(v2) + (270:5*\vrs + \vrtwa)$) arc (270:450:5*\vrs + \vrtwa) to[out=180,in=0]
            ($(v3) + (90:\vrs + \vrtwa)$) arc (90:270:\vrs +\vrtwa) to[out=0,in=180] cycle;
\draw[edgesphere] ($(v3) + (330:5*\vrs + \vrtwa)$) arc (-30:150:5*\vrs + \vrtwa) to[out=240,in=60]
            ($(v4) + (150:\vrs + \vrtwa)$) arc (150:330:\vrs +\vrtwa) to[out=60,in=240] cycle;
\draw[edgesphere] ($(v4) + (30:5*\vrs + \vrtwa)$) arc (30:210:5*\vrs + \vrtwa) to[out=300,in=120]
            ($(v5) + (210:\vrs + \vrtwa)$) arc (210:390:\vrs +\vrtwa) to[out=120,in=300] cycle;
\draw[edgesphere] ($(v5) + (90:5*\vrs + \vrtwa)$) arc (90:270:5*\vrs + \vrtwa) to[out=0,in=180]
            ($(v6) + (270:\vrs + \vrtwa)$) arc (270:450:\vrs +\vrtwa) to[out=180,in=0] cycle;
\draw[edgesphere] ($(v6) + (150:5*\vrs + \vrtwa)$) arc (150:330:5*\vrs + \vrtwa) to[out=60,in=240]
            ($(v1) + (330:\vrs + \vrtwa)$) arc (-30:150:\vrs +\vrtwa) to[out=240,in=60] cycle;
\draw[chordsphere] 
    ($(v1) + (-120:4*\vrs + \vrtwa)$) arc  (-120:60:4*\vrs + \vrtwa) to[out=150,in=330]
    ($(v3) + (60:2*\vrs + \vrtwa)$) arc (60:240:2*\vrs +\vrtwa) to[out=330,in=150] cycle;
\draw[chordsphere] 
    ($(v2) + (-60:4*\vrs + \vrtwa)$) arc  (-60:120:4*\vrs + \vrtwa) to[out=210,in=30]
    ($(v4) + (120:2*\vrs + \vrtwa)$) arc (120:300:2*\vrs +\vrtwa) to[out=30,in=210] cycle;
\draw[chordsphere] 
    ($(v3) + (0:4*\vrs + \vrtwa)$) arc  (0:180:4*\vrs + \vrtwa) to[out=270,in=90]
    ($(v5) + (180:2*\vrs + \vrtwa)$) arc (180:360:2*\vrs +\vrtwa) to[out=90,in=270] cycle;
\draw[chordsphere] 
    ($(v4) + (60:4*\vrs + \vrtwa)$) arc  (60:240:4*\vrs + \vrtwa) to[out=330,in=150]
    ($(v6) + (240:2*\vrs + \vrtwa)$) arc (-120:60:2*\vrs +\vrtwa) to[out=150,in=330] cycle;
\draw[chordsphere] 
    ($(v5) + (120:4*\vrs + \vrtwa)$) arc  (120:300:4*\vrs + \vrtwa) to[out=30,in=210]
    ($(v1) + (-60:2*\vrs + \vrtwa)$) arc (-60:120:2*\vrs +\vrtwa) to[out=210,in=30] cycle;
\draw[chordsphere] 
    ($(v6) + (180:4*\vrs + \vrtwa)$) arc (180:360:4*\vrs + \vrtwa) to[out=90,in=270]
    ($(v2) + (0:2*\vrs + \vrtwa)$) arc (0:180:2*\vrs +\vrtwa) to[out=270,in=90] cycle;

\draw[diametersphere]
    ($(v1) + (-90:3*\vrs + \vrtwa)$) arc (-90:90:3*\vrs + \vrtwa) --
    ($(v4) + (90:3*\vrs + \vrtwa)$) arc (90:270:3*\vrs + \vrtwa) -- cycle;
\draw[diametersphere]
    ($(v2) + (-30:3*\vrs + \vrtwa)$) arc (-30:150:3*\vrs + \vrtwa) --
    ($(v5) + (150:3*\vrs + \vrtwa)$) arc (150:330:3*\vrs + \vrtwa) -- cycle;
\draw[diametersphere]
    ($(v3) + (30:3*\vrs + \vrtwa)$) arc (30:210:3*\vrs + \vrtwa) --
    ($(v6) + (210:3*\vrs + \vrtwa)$) arc (210:390:3*\vrs + \vrtwa) -- cycle;
    
\draw[Acolor] (v3) to[out=230,in=60] 
    ($(v4) + (150:\vrs + \vrga)$) arc (150:330:\vrs+\vrga) 
    to[out=60,in=250] (v3);
\draw[Acolor] (v2) to[out=260,in=90]
    ($(v6) + (180:2*\vrs + \vrga)$) arc (180:360:2*\vrs + \vrga)
    to[out=90,in=280] (v2);
\draw[Acolor!60] (v3) to[out=260,in=90] 
    ($(v5) + (180:2*\vrs + \vrga)$) arc (180:360:2*\vrs+\vrga) 
    to[out=90,in=280] (v3);
\draw[Acolor!60] (v2) to[out=290,in=120]
    ($(v1) + (210:\vrs + \vrga)$) arc (210:390:\vrs + \vrga)
    to[out=120,in=310] (v2);
\draw[Bcolor] (v5) to[out=350,in=180] 
    ($(v6) + (270:\vrs + \vrga)$) arc (270:450:\vrs+\vrga) 
    to[out=180,in=10] (v5);
\draw[Bcolor] (v4) to[out=20,in=210]
    ($(v2) + (300:2*\vrs + \vrga)$) arc (-60:120:2*\vrs + \vrga)
    to[out=210,in=40] (v4);
\draw[Bcolor!70!black] (v5) to[out=20,in=210] 
    ($(v1) + (300:2*\vrs + \vrga)$) arc (-60:120:2*\vrs+\vrga) 
    to[out=210,in=40] (v5);
\draw[Bcolor!70!black] (v4) to[out=50,in=240]
    ($(v3) + (330:\vrs + \vrga)$) arc (-30:150:\vrs + \vrga)
    to[out=240,in=70] (v4);
\draw[Ccolor] (v1) to[out=110,in=300] 
    ($(v2) + (30:\vrs + \vrga)$) arc (30:210:\vrs+\vrga) 
    to[out=300,in=130] (v1);
\draw[Ccolor] (v6) to[out=140,in=330]
    ($(v4) + (60:2*\vrs + \vrga)$) arc (60:240:2*\vrs + \vrga)
    to[out=330,in=160] (v6);
\draw[Ccolor!60] (v1) to[out=140,in=330] 
    ($(v3) + (60:2*\vrs + \vrga)$) arc (60:240:2*\vrs+\vrga) 
    to[out=330,in=160] (v1);
\draw[Ccolor!60] (v6) to[out=170,in=0]
    ($(v5) + (90:\vrs + \vrga)$) arc (90:270:\vrs + \vrga)
    to[out=0,in=190] (v6);
\end{tikzpicture} \end{center}
\end{minipage}
\hspace*{-2.3in}
\caption{A view of the rigid sphere system for $M_{3,0}$, rendered in  $M_{0,6}$ with
boundary sphere identifications. The good pair disks are color-coded to indicate the identification.}
\end{figure}

\section{Exhaustion by strongly rigid sets}\label{rigid exhaustion section}

We move from one geometrically rigid set to an exhaustion of $\Sc(M_{n,0})$ by such sets by developing a notion of rigid expansion that parallels Hernández Hernández' construction in the surface setting~\cite{hernandezhernandez-exhaustion}. 
In \cref{strong rigidity section} we will conclude that each set in the exhaustion is in fact strongly rigid.

\begin{definition}\label{pants sphere terminology}
Suppose $P$ is a pants decomposition of $M_{n,0}$.
Two spheres $a, b \in P$ are \emph{adjacent in $P$} if they are two of the boundary spheres of some pair of pants component of $M_{n,0} \setminus P$. A sphere $a\in P$ is \emph{self-adjacent in $P$} if $a$ bounds two cuffs of a single pair of pants in $M_{n,0}\setminus P$.
\end{definition}

\begin{definition}\label{split sphere for pants}
Suppose $P$ is a pants decomposition of $M_{n,0}$ and $a \in P$. A sphere $b \in \Sc(M_{n,0})$ is a \emph{split sphere} for $(a,P)$ if $a$ is the unique sphere in $P$ intersecting $b$.  

If $X\subseteq \Sc(M_{n,0})$ is a subcomplex, $P \subseteq X^{(0)}$, and $b \in X^{(0)}$ is a split sphere for $(a,P)$, then we say that $P$ is \emph{$X$--split at $a$ (by $b$)}.  We  say that $P$ is \emph{$X$--split} if it is $X$--split at $a$ for some $a \in P$.  If $X$ contains every split sphere for $P \subset X$, then we say that $P$ is \emph{fully $X$--split}.
\end{definition}

Observe that if $P$ is a pants decomposition and $a\in P$, a split sphere for $(a,P)$ exists if and only if $a$ is contained in an $M_{0,4}$ component of the complement of $P\setminus\{a\}$, that is, $a$ is not self-adjacent in $P$. In this case there are exactly two split spheres for $P$ intersecting $a$, by \cref{partition-determines}. Since there are two such spheres, we cannot guarantee that adding a single split sphere results in a rigid set. However for certain pairs we can exploit \cref{six-holed-evil-twins-cross}.

\begin{definition}\label{a split pair for a sphere}
Suppose $X\subseteq \Sc(M_{n,s})$ is a subcomplex and $a \in X^{(0)}$. A pair of distinct, disjoint spheres $(b_1,b_2)$ in $\Sc(M_{n,s})^{(0)}$ is a \emph{split pair for $a$} if there exists pants decompositions $P_1, P_2 \subseteq X^{(0)}$, both containing $a$, such that $b_i$ is a split sphere for $(a,P_i)$, for $i=1,2$.
\end{definition}

\begin{lemma}\label{adding a split pair is rigid}
Suppose $X\subseteq \Sc(M_{n,0})$ is a geometrically rigid subcomplex, and $a\in X^{(0)}$ has a split pair $(b, c)$. Then the subcomplex $X_{b,c}$ induced by $X\cup \{b, c\}$ is geometrically rigid.
\end{lemma}

\begin{proof}
Observe that both the intersection of $a$ with $b$ and $a$ with $c$ is $X_{b,c}$-detectable, using the pants decompositions $P_b$ and $P_c$ witnessing the split pair.
Now suppose $f \colon X_{b,c} \to \Sc(M_{n,0})$ is a locally injective simplicial map. Since $X$ is geometrically rigid there is a homeomorphism $h \colon M_{n,0} \to M_{n,0}$ such that in the induced map on the sphere complex $h|_X = f|_X$. By \cref{detect-to-detect}, $f(a)$ and $f(b)$ have $f(X_{b,c})$-detectable intersection, so $f(b)$ is a sphere distinct from $f(a)$ in $N(f(a)\cup f(b))$. Similarly $f(c)$ is a sphere distinct from $f(a)$ in $N(f(a)\cup f(c))$. Further, $h(P_b) = f(P_b)$ and $h(a) = f(a)$, so $h(N(a\cup b)) = N(f(a)\cup f(b))$, and similarly $h(N(a\cup c)) = N(f(a)\cup f(c))$. Let $b'$ be the sphere in $\Sc(N(a\cup b))$ other than $a$ and $b$ and define $c'$ similarly. By \cref{six-holed-evil-twins-cross}, $b'$ and $c'$ intersect and both intersect $b$ and $c$, so the same is true of $h(b')$ and $h(c')$. Since $b$ and $c$ are disjoint and $f$ is locally injective and simplicial we conclude $h(b) = f(b)$ and $h(c) = f(c)$ as required.
\end{proof}

From the lemma (and the notation from the proof), we see that $P_b$ is $X_{b,c}$--split at $a$ by $b$ (and similarly for $P_c$).  We also record the following obvious fact:

\begin{observation}\label{split-exchange X split}
If $P\subseteq X^{(0)}$ is a pants decomposition that is $X$-split at $a$ by $b\in X^{(0)}$, then the pants decomposition $P' = (P\setminus \{a\})\cup\{b\}$ is $X$-split at $b$ by $a$.
\end{observation}

\begin{lemma}\label{split pairs exist}
Suppose $X\subseteq \Sc(M_{n,0})$ is a subcomplex and $P\subseteq X^{(0)}$ is a pants decomposition that is $X$--split at $a$ by $b \in X^{(0)}$. For every sphere $c\in P$ adjacent to $a$, if $c$ has a split sphere $d$ then there is a sphere $e$ such that $(d,e)$ is a split pair for $c$.
\end{lemma}

\begin{proof}
Fix a sphere $c\in P$ adjacent to $a$ which has a split sphere, i.e. so that $M_{n,0}\setminus \{P\setminus \{c\}\}$ has an $M_{0,4}$ connected component. 
Let $N$ be the connected component of $M_{n,0} \setminus \{P\setminus {a,c}\}$ containing $a$ and $c$. 
Since both $a$ and $c$ have split spheres, $N\simeq M_{0,5}$ and we identify $\partial M_{0,5}$ with $[5]$. Further, we identify the spheres $a$ and $c$ with size 2 subsets of $[5]$ by \cref{partition-determines},
choosing labels so that $a = \{1,2\}$, $b = \{2,3\}$ and $c = \{4, 5\}$. There are two possible split spheres for $c\in P$---the spheres $d_4 = \{3, 4\}$ and $d_5 = \{3, 5\}$---and to complete the proof of the lemma we will
produce their split pairs.  Let $P' = (P\setminus\{a\}) \cup\{b\}$. 
For $d_4$, the sphere $e_4 = \{1, 5\}$ is a split sphere for $c \in P'$, and $(d_4, e_4)$ is a split pair. For $d_5$, the sphere $e_5 = \{1,4\}$ is a split sphere for $c\in P'$ and $(d_5, e_5)$ is a split pair for $c$ in $P'$. See Figure \ref{building split pairs from a split sphere} for an illustration.
\end{proof}

\begin{figure}[ht]
\begin{tikzpicture}[every node/.style={font=\scriptsize},thick,scale=2]
\node[circle,draw=black,fill=black!10,thick] (v1) at (0:1) {$1$};
\node[circle,draw=black,fill=black!10,thick] (v2) at (72:1) {$2$};
\node[circle,draw=black,fill=black!10,thick] (v3) at (144:1) {$3$};
\node[circle,draw=black,fill=black!10,thick] (v4) at (216:1) {$4$};
\node[circle,draw=black,fill=black!10,thick] (v5) at (288:1) {$5$};

\draw[green] ($ (v1)+(216:0.2) $) arc (216:396:0.2) to[out=126,in=306]  node[midway,below left] {$a$}
                    ($ (v2)+(36:0.2) $) arc  (36:216:0.2) to[out=306,in=126]
                   cycle;
\draw[yellow] ($(v2) + (288:0.25)$) arc (288:468:0.25) to[out=198,in=18] node[midway,below right] {$b$}
            ($(v3) + (108:0.25)$) arc (108:288:0.25) to[out=18,in=198] cycle;
\draw[green] ($(v1) + (252:0.3)$) arc (252:396:0.3) to[out=126,in=306] node[midway,above right] {$c$}
            ($(v2) + (36:0.3)$) arc (36:108:0.3) to[out=198,in=18]
            ($(v3) + (108:0.3)$) arc (108:252:0.3) to[out=342,in=162] cycle;
            
\draw[blue] ($(v3) + (0:0.35)$) arc (0:180:0.35) to[out=270,in=90]
              node[midway,left] {$d_4$}
              ($(v4) + (180:0.35)$) arc (180:360:0.35) to[out=90,in=270]
              cycle;
\draw[cyan] ($(v1) + (324:0.35)$) arc (-36:144:0.35) to [out=234,in=54]
            ($(v5) + (144:0.35)$) arc (144:324:0.35) to [out=54, in=234] node[midway, below right] {$e_4$}
            cycle;
\draw[red] ($(v3) + (36:0.4)$) arc (36:216:0.4) to [out=306,in=126]
              ($(v5) + (216:0.4)$) arc (216:396:0.4) node[midway, below] {$d_5$} to [out=126,in=306]
              cycle;
\draw[orange] ($(v1) + (288:0.4)$) arc (288:468:0.4) to[out=198,in=18]
              ($(v4) + (108:0.4)$) arc (108:288:0.4) node[midway,below left] {$e_5$} to[out=18,in=198]
              cycle;
\end{tikzpicture} \caption{The split pairs $(d_4, e_4)$ and $(d_5, e_5)$ for the sphere $c$ constructed in \cref{split pairs exist}. The pairs satisfy the definition of split pair because the sphere $b$ is in $X^{(0)}$ by hypothesis.}\label{building split pairs from a split sphere}
\end{figure}
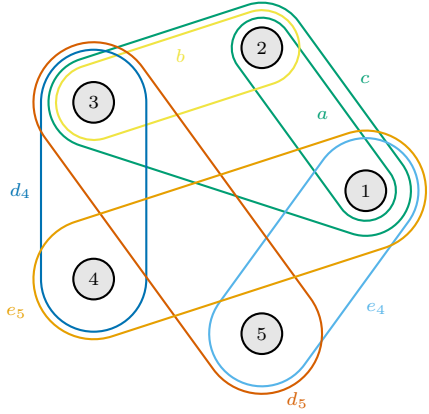

\begin{lemma}\label{pants splitting expansion}
Suppose $X\subseteq \Sc(M_{n,0})$ is a finite geometrically rigid set and $P\subseteq X^{(0)}$ is $X$--split.  Then there is a finite geometrically rigid set $X^P \supset X$ so that $P$ is fully $X^P$--split; that is,  $X^P$ contains every split sphere for $P$. \end{lemma}

\begin{proof}
Let $a_0 \in P$ be a sphere of $P$ witnessing that $P$ is $X$-split.
Inductively define $P_0 = \{a_0\}$ and 
\[ P_i = \{ s \in P |\text{ $s$ is adjacent to $a\in P_{i-1}$ and not self-adjacent } \}. \]
Note that there exists some $k$ such that $\displaystyle{\cup_{i=1}^k P_i}$ is all non-self-adjacent spheres in $P$.  To see this, consider the dual (3-valent) graph to $P$, which has a vertex for every pair of pants and edge connecting pants that share a boundary sphere.  Then the spheres in $P$ which have a split sphere correspond precisely to the non-loop edges, which forms a connected subgraph of the dual.

Next we inductively define finite geometrically rigid subcomplexes $X_i \supset X$ such that for each $a\in P_i$, the pants decomposition $P$ is $X_i$ split at $a$ and contains both split spheres for $(a,P)$.  Since $P$ is $X$--split at $a$ by some $b \in X^{(0)}$, there is at most one other sphere $c$ that is a split sphere for $(a,P)$, and we set $X_0$ to be the subcomplex induced by $X \cup \{c\}$. It is straightforward to see that $X_0$ has the desired property.

We now suppose we have constructed $X_{i-1}$ in this fashion and construct $X_i$.
 For each sphere $c\in P_i$ there are two split spheres $d, d'$ for $(c,P)$. The sphere $c$ is adjacent to $a\in P_{i-1}$, and $P$ is $X_{i-1}$--split at $a$, so by \cref{split pairs exist} there exist split pairs $(d_c, e_c)$ and $(d'_c, e'_c)$ for $c$. Let
$X_i$ be the subcomplex induced by $X_{i-1}$ and $\{d_c,e_c,d'_c,e'_c\}_{c\in P_i}$ for the split pairs constructed in the previous sentence as $c$ ranges over $P_i$. Repeated application of \cref{adding a split pair is rigid} implies that $X_i$ is geometrically rigid, and for each $c\in P_i$, the pants decomposition is $X_i$-split at $c$ by construction. Moreover, $X_i$ contains the split spheres of $P$ intersecting each $c\in P_i$. Take $X^P = X_k$. This is a finite geometrically rigid complex in $\Sc(M_{n.0})$ by construction; moreover, every split sphere $s$ for $P$ intersects a unique non-self adjacent sphere $c_s\in P$, and $c_s\in P_i$ for some $i\leq k$, hence $s\in X_i \subseteq X^P$.
\end{proof}

\begin{proposition}\label{geometrically rigid exhaustion}
There exists a nested family of finite geometrically rigid sets $X_j \subseteq \Sc(M_{n,0})$ such that
\[ \Sc(M_{n,0}) = \cup_j X_j. \]
\end{proposition}

\begin{proof}
Let $X$ be the finite strongly rigid set constructed in \cref{geometrically rigid set}. By construction $X$ contains a pants decomposition $P_0$ that is $X$-split.

Define a sequence of sets of pants decompositions of $M_{n,0}$ as follows. Begin with $\mathcal{P}_0 = \{ P_0 \}$ and define
\[ \mathcal{P}_i = \{ P\text{ a pants decomposition } | \text{ there is a $P' \in \mathcal{P}_{i-1}$ such that $|P\Delta P'| = 2$ }\}. \]
Observe that if $P\in \mathcal{P}_i$ then $P$ is obtained from $P'$ by exchanging a split sphere. Thus each $\mathcal{P}_i$ is finite. Moreover, a porism of Hatcher's proof of the connectivity of $\Sc(M_{n,0})$~\cite{hatcher}*{Theorem 2.1} is that every pants decomposition can be obtained from a reference pants decomposition by exchanging finitely many split spheres, and thus every pants decomposition appears in some $\mathcal{P}_k$.

Now define a sequence of subcomplexes $X_i$ such that for each pants decomposition $P\in \mathcal{P}_{i+1}$, both $P\subseteq X_i^{(0)}$ and $P$ is $X_i^{(0)}$-split. Start with $X_0 = X^{P_0}$, the complex obtained by applying \cref{pants splitting expansion} to $X$ and $P_0$.  The conclusions of that lemma guarantee that each $P\in\mathcal{P}_1$ is contained in $X_0^{(0)}$ and is $X_0$-split. 

For each $i$ fix an enumeration of $\mathcal{P}_i = \{ P_{i,1},\ldots, P_{i,k_i} \}$, and set $X_{i,0} = X_{i-1}$. Define $X_{i,j} = (X_{i,j-1})^{P_{i,j}}$ the complex obtained by applying \cref{pants splitting expansion} to $X_{i,j-1}$ and $P_{i,j}$. Finally set $X_i = X_{i,k_i}$, this is geometrically rigid by induction. 

Every pants decomposition $P\in \mathcal{P}_{i+1}$ is obtained from some $P_{i,j}\in \mathcal{P}_i$ by exchanging a split sphere, and so it follows from \cref{pants splitting expansion} that $P\subseteq X_{i,j}\subseteq X_i$. By \cref{split-exchange X split}, $P$ is $X_{i,j}$-split, and therefore it is $X_i$-split. 

Every sphere appears in some pants decomposition, hence every sphere appears as a member of some $P_{i,j}$, whence:
\[ \Sc(M_{n,0}) = \bigcup_{k} X_k.\qedhere \]
\end{proof}

\section{From geometric rigidity to strong rigidity}\label{strong rigidity section}

We now upgrade geometric rigidity to strong rigidity, in two steps. First, we show that if a mapping class fixes the isotopy classes of a pants decomposition pointwise, then it induces the identity on the sphere complex.

\begin{lemma}
Suppose $h \in \Mod(M_{n,0})$ is in the point-wise stabilizer of a pants decomposition $P \subseteq \Sc(M_{n,0})$ such that each sphere is the boundary of two distinct complementary components. Then $h$ induces the identity on $\Sc(M_{n,0})$.
\end{lemma}

\begin{proof}
If $h$ fixes each sphere of a pants decomposition up to isotopy, then $h$ fixes each complementary component up to isotopy (since $n \geq 3$). Since each sphere is the boundary of two distinct connected components, the restriction of $h$ to each pair of pants fixes the boundary components. Thus, restricted to each pair of pants $h$ is isotopic to the identity, and differs from the identity by Dehn twisting in the pants spheres $P$. Since Dehn twists generate the kernel of the map $\Mod(M_{n,0})\to\Sc(M_{n,0})$~\cite{laudenbach} we are done.
\end{proof}

\begin{corollary}\label{C:geometric rigidity pants}
If $X\subseteq \Sc(M_{n,0})$ is geometrically rigid and contains a pants decomposition where each sphere is the boundary of two distinct complementary components then it is uniquely geometrically rigid, in the sense that any mapping class that stabilizes $X$ pointwise induces the identity on $\Sc(M_{n,0})$.
\end{corollary}

Applying this corollary to the exhaustion constructed in the previous section we give a new proof of the following theorem of Aramayona and Souto. We complete the proof of our first two main theorems using this isomorphism.

\begin{theorem}[\cite{aramayona-souto}]\label{global rigidity}
When $n\geq 3$, $\Out(F_n) \cong \Aut(\Sc(M_{n,0}))$.
\end{theorem}

\begin{proof}[New proof]
Consider $\phi \in \Aut(\Sc(M_{n,0}))$. Let $X_j$ be the exhaustion of $\Sc(M_{n,0})$ by geometrically rigid sets constructed in \cref{geometrically rigid exhaustion}. For each $j$ we get a homeomorphism $h_j$ such that $\phi|_{X_j} = h_j|_{X_j}$. Moreover, by construction each $X_j$ contains a common pants decomposition where each sphere is the boundary of two distinct complementary components. Therefore, it follows from \cref{C:geometric rigidity pants} that for $j \neq k$ the homeomorphisms $h_j$ and $h_k$ induce the same automorphism of $\Sc(M_{n,0})$. Hence $\phi|_{X_j} = h_1|_{X_j}$ for all $j$, and we conclude that $\phi = h_1$. Hence $\Mod(M_{n,0}) \to \Aut(\Sc(M_{n,0}))$ is surjective.

To conclude, it is a consequence of Laudenbach's theorem~\cite{laudenbach}*{Théorème III} that this homomorphism factors through the surjective map $\Mod(M_{n,0}) \to \Out(F_n)$ and that the map $\Out(F_n)\to \Aut(\Sc(M_{n,0})$ is injective.
\end{proof}

\begin{proof}[Proof of \cref{rigid set theorem} and \cref{rigid exhaustion theorem}]
It follows from \cref{global rigidity} that the pointwise stabilizer of $\Aut(\Sc(M_{n,0}))$ of a set $X$ is equal to the image of its $\Mod(M_{n,0})$ stabilizer.

The set $X$ constructed in \cref{geometrically rigid set} contains a pants decomposition, so by \cref{C:geometric rigidity pants} the image of its $\Mod(M_{n,0})$ stabilizer is the identity. This completes the proof of \cref{rigid set theorem}.

Since $X$ is a subset of every set constructed in \cref{geometrically rigid exhaustion}, each set in the exhaustion contains a pants, and \cref{rigid exhaustion theorem} follows from an identical argument.
\end{proof}

\section{No rigid sets in rank 2} \label{S:no rigid rank 2}

The situation in rank $2$ is quite different.

\begin{proposition}
There is no finite rigid set $X \subset \Sc(M_{2,0})$.
\end{proposition}
\begin{proof}
The graph $\Sc(M_{2,0})$ is obtained from the Farey graph, $\mathcal F$, by adding an edge path of length two between every two adjacent vertices of $\mathcal F$.  That is, we attach (the $1$--skeleton of) a triangle--sometimes called a {\em fin}--to every edge of $\mathcal F$ along an edge; see Culler and Vogtmann's description of outer space in rank 2~\cite{culler-vogtmann-farey-fins} and Hatcher~\cite{hatcher}*{Appendix} for the translation to the sphere complex. Each vertex of $\mathcal F$ has infinite valence while the new vertices each have valence $2$.

Every element of $\Aut(\mathcal F) \cong \PGL_2(\mathbb Z)$ has a unique simplicial extension to an automorphism of $\Sc(M_{2,0})$, defining a natural isomorphism $\Aut(\Sc(M_{2,0})) \cong \Aut(\mathcal F)$.  In particular, there are two orbits of edges: the Farey edges and the added {\em fin edges}; and two orbits of vertices: the infinite valence vertices and valence $2$ vertices.  

We suppose $X \subset \Sc(M_{2,0})$ is any finite subgraph, and show that it cannot be rigid.  Since a disconnected graph cannot be rigid, we may assume $X$ is connected.  Furthermore, it must consist of more than one edge (since there are two orbits of edges).  Thus, we may assume $X$ is a connected subgraph with at least $2$ edges.

Now suppose that $X$ has a valence $1$ vertex, $v$.  This vertex must be one endpoint of an edge $e$ in $X$, and we let $u$ denote the other endpoint.  If $u$ has infinite valence in $\Sc(M_{2,0})$, then there are infinitely many simplicial embeddings of $X$ into $\Sc(M_{2,0})$ that are all the identity on $X\setminus e$ and which send $e$ to distinct edges of $\Sc(M_{2,0})$.  Since there are only two elements of $\PGL_2(\mathbb Z)$ that act as the identity on any given Farey edge, at most two of the embeddings are restrictions of elements of $\Aut(\Sc(M_{2,0}))$.  Any other embedding is therefore not the restriction of an element of $\Aut(\Sc(M_{2,0}))$ and so $X$ is not rigid.  

If $u$ is a valence $2$ vertex, then since $v$ is valence $1$, connectivity of $X$ ensures that there is a second edge $e'$ so that $e \cup e'$ is a length two path in $X$.  In this case, note that both $e$ and $e'$ are fin edges.  If $X = e \cup e'$, this is clearly not rigid since we can map this to any length two path in the Farey graph, which is thus not the restriction of an element of $\Aut(\Sc(M_{2,0}))$.  Therefore, $e \cup e'$ is a length two path in $X$ which meets the rest of $X$ in a single vertex $w$ (the other endpoint of $e'$).  Since $w$ must be infinite valence, we can find infinitely many distinct simplicial embeddings of $X$ which are the identity on $X\setminus (e\cup e')$, and again at most two of these can be restrictions of elements of $\Aut(\Sc(M_{2,0}))$, so $X$ is not rigid.

By the arguments above, we may now assume that $X$ has no valence $1$ vertices.  Let $X_{\mathcal F}$ denote the convex hull of $X\cap \mathcal F$ in $\mathcal F$.  Observe that $X_{\mathcal F}$ is (the $1$--skeleton of) a triangulation of a polygon with vertices on the boundary, and thus has a valence $2$ vertex, $v$.  Let $e_1$ and $e_2$ denote the edges of $X_{\mathcal F}$ adjacent to $v$, and $u$ and $w$, respectively, denote their other endpoints.  Then $u$ and $w$ are vertices of an edge $e_3$ which together with $e_1$ and $e_2$ defines a triangle in $X_{\mathcal F}$.

Now suppose $e_1 \cup e_2 \subset X$ and let $e \cup e'$ be the length two path of fin edges connecting $u$ and $w$.  Either both $e$ and $e'$ are in $X$ or neither is, since $X$ has no valence $1$ vertices.  In either case, there is an embedding of $X$ that maps $e_1 \cup e_2$ to $e \cup e'$ and is the identity on $X\setminus (e_1 \cup e_2 \cup e \cup e')$ (swapping $e_1 \cup e_2$ with $e \cup e'$ if the latter is contained in $X$).  Since this sends Farey edges to fin edges, this cannot be the restriction of an element of $\Aut((M_{2,0}))$, hence $X$ is not rigid.

Therefore, we may assume that one or both of $e_1$ and $e_2$ are not in $X$.  In either case, since $X_{\mathcal F}$ was the convex hull of $X \cap \mathcal F$, it follows that $v$ is a vertex of $X$.  Therefore, there is a length two path of fin edges $e \cup e'$ in $X$ having $v$ as an endpoint, and the other endpoint is either $u$ or $w$.  In either case, there is a length two path $e_4 \cup e_5$ of Farey edges connecting the endpoints of $e \cup e'$, that meets $X_{\mathcal F}$, and hence $X$, only in the endpoints.  From this we can construct a simplicial embedding of $X$ which is the identity on $X \setminus (e \cup e')$ and sends $e \cup e'$ to $e_4 \cup e_5$.  Since this sends fin edges to Farey edges, this cannot be the restriction of an element of $\Aut((\Sc(M_{2,0}))$, and hence $X$ is not rigid.

The cases above exhaust all possibilities for $X$, and in all cases $X$ fails to be rigid.  This completes the proof.
\end{proof}


\begin{bibdiv}
\begin{biblist}
\bib{aramayona-souto}{article}{
  author={Aramayona, Javier},
  author={Souto, Juan},
  title={Automorphisms of the graph of free splittings},
  journal={Michigan Math. J.},
  volume={60},
  date={2011},
  number={3},
  pages={483--493},
  issn={0026-2285},
  review={\MR {2861084}},
  doi={10.1307/mmj/1320763044},
}

\bib{aramayona-leininger}{article}{
  author={Aramayona, Javier},
  author={Leininger, Christopher J.},
  title={Finite rigid sets in curve complexes},
  journal={J. Topol. Anal.},
  volume={5},
  date={2013},
  number={2},
  pages={183--203},
  issn={1793-5253},
  review={\MR {3062946}},
  doi={10.1142/S1793525313500076},
}

\bib{aramayona-leininger-2}{article}{
  author={Aramayona, Javier},
  author={Leininger, Christopher J.},
  title={Exhausting curve complexes by finite rigid sets},
  journal={Pacific J. Math.},
  volume={282},
  date={2016},
  number={2},
  pages={257--283},
  issn={0030-8730},
  review={\MR {3478935}},
  doi={10.2140/pjm.2016.282.257},
}

\bib{brendle-margalit}{article}{
  author={Brendle, Tara E.},
  author={Margalit, Dan},
  title={Normal subgroups of mapping class groups and the metaconjecture of Ivanov},
  journal={J. Amer. Math. Soc.},
  volume={32},
  date={2019},
  number={4},
  pages={1009--1070},
  issn={0894-0347},
  review={\MR {4013739}},
  doi={10.1090/jams/927},
}

\bib{disarlo-koberda-Gonzalez}{article}{
  title={The model theory of the curve graph},
  author={Disarlo, Valentina},
  author={Koberda, Thomas},
  author={González, Javier de la Nuez},
  year={2020},
  eprint={arxiv:2008.10490},
}

\bib{culler-vogtmann-farey-fins}{article}{
  author={Culler, Marc},
  author={Vogtmann, Karen},
  title={The boundary of outer space in rank two},
  conference={ title={Arboreal group theory}, address={Berkeley, CA}, date={1988}, },
  book={ series={Math. Sci. Res. Inst. Publ.}, volume={19}, publisher={Springer, New York}, },
  date={1991},
  pages={189--230},
  review={\MR {1105335}},
}

\bib{hatcher}{article}{
  author={Hatcher, Allen},
  title={Homological stability for automorphism groups of free groups},
  journal={Comment. Math. Helv.},
  volume={70},
  date={1995},
  number={1},
  pages={39--62},
  issn={0010-2571},
  review={\MR {1314940}},
}

\bib{hernandezhernandez-exhaustion}{article}{
  author={Hern\'{a}ndez Hern\'{a}ndez, Jes\'{u}s},
  title={Exhaustion of the curve graph via rigid expansions},
  journal={Glasg. Math. J.},
  volume={61},
  date={2019},
  number={1},
  pages={195--230},
  issn={0017-0895},
  review={\MR {3882310}},
  doi={10.1017/S0017089518000174},
}

\bib{hhlm-pants}{article}{
  author={Hern\'{a}ndez Hern\'{a}ndez, Jes\'{u}s},
  author={Leininger, Christopher J.},
  author={Maungchang, Rasimate},
  title={Finite rigid subgraphs of pants graphs},
  journal={Geom. Dedicata},
  volume={212},
  date={2021},
  pages={205--223},
  issn={0046-5755},
  review={\MR {4251670}},
  doi={10.1007/s10711-020-00555-1},
}

\bib{ik-nonorientable}{article}{
  author={Ilbira, Sabahattin},
  author={Korkmaz, Mustafa},
  title={Finite rigid sets in curve complexes of nonorientable surfaces},
  journal={Geom. Dedicata},
  volume={206},
  date={2020},
  pages={83--103},
  issn={0046-5755},
  review={\MR {4091539}},
  doi={10.1007/s10711-019-00478-6},
}

\bib{irmak-superinjective}{article}{
  author={Irmak, Elmas},
  title={Superinjective simplicial maps of complexes of curves and injective homomorphisms of subgroups of mapping class groups},
  journal={Topology},
  volume={43},
  date={2004},
  number={3},
  pages={513--541},
  issn={0040-9383},
  review={\MR {2041629}},
  doi={10.1016/j.top.2003.03.002},
}

\bib{irmak-nonorientable}{article}{
  author={Irmak, Elmas},
  title={Exhausting curve complexes by finite superrigid sets on nonorientable surfaces},
  journal={Fund. Math.},
  volume={255},
  date={2021},
  number={2},
  pages={111--138},
  issn={0016-2736},
  review={\MR {4319609}},
  doi={10.4064/fm835-3-2021},
}

\bib{irmak-nonorientable-2}{article}{
  author={Irmak, Elmas},
  title={Exhausting curve complexes by finite sets on nonorientable surfaces},
  journal={J.~{T}op.~{A}nal.},
  doi={10.1142/s1793525321500679},
  date={2022},
  pages={1--29},
}

\bib{ivanov}{article}{
  author={Ivanov, Nikolai V.},
  title={Automorphism of complexes of curves and of Teichm\"{u}ller spaces},
  journal={Internat. Math. Res. Notices},
  date={1997},
  number={14},
  pages={651--666},
  issn={1073-7928},
  review={\MR {1460387}},
  doi={10.1155/S1073792897000433},
}

\bib{whitney-theorem-ref}{article}{
  author={Jung, H. A.},
  title={Zu einem Isomorphiesatz von H. Whitney f\"{u}r Graphen},
  language={German},
  journal={Math. Ann.},
  volume={164},
  date={1966},
  pages={270--271},
  issn={0025-5831},
  review={\MR {197353}},
  doi={10.1007/BF01360250},
}

\bib{korkmaz}{article}{
  author={Korkmaz, Mustafa},
  title={Automorphisms of complexes of curves on punctured spheres and on punctured tori},
  journal={Topology Appl.},
  volume={95},
  date={1999},
  number={2},
  pages={85--111},
  issn={0166-8641},
  review={\MR {1696431}},
  doi={10.1016/S0166-8641(97)00278-2},
}

\bib{laudenbach}{article}{
  author={Laudenbach, F.},
  title={Sur les $2$-sph\`eres d'une vari\'{e}t\'{e} de dimension $3$},
  language={French},
  journal={Ann. of Math. (2)},
  volume={97},
  date={1973},
  pages={57--81},
  issn={0003-486X},
  review={\MR {314054}},
  doi={10.2307/1970877},
}

\bib{luo}{article}{
  author={Luo, Feng},
  title={Automorphisms of the complex of curves},
  journal={Topology},
  volume={39},
  date={2000},
  number={2},
  pages={283--298},
  issn={0040-9383},
  review={\MR {1722024}},
  doi={10.1016/S0040-9383(99)00008-7},
}

\bib{Mate2}{article}{
  author={Maungchang, Rasimate},
  title={Exhausting pants graphs of punctured spheres by finite rigid sets},
  journal={J. Knot Theory Ramifications},
  volume={26},
  date={2017},
  number={14},
  pages={1750105, 11},
  issn={0218-2165},
  review={\MR {3735406}},
  doi={10.1142/S021821651750105X},
}

\bib{Mate}{article}{
  author={Maungchang, Rasimate},
  title={Finite rigid subgraphs of the pants graphs of punctured spheres},
  journal={Topology Appl.},
  volume={237},
  date={2018},
  pages={37--52},
  issn={0166-8641},
  review={\MR {3760775}},
  doi={10.1016/j.topol.2018.01.009},
}

\bib{McLeay}{article}{
  author={McLeay, Alan},
  title={Geometric normal subgroups in mapping class groups of punctured surfaces},
  journal={New York J. Math.},
  volume={25},
  date={2019},
  pages={839--888},
  review={\MR {4012570}},
}

\bib{shackleton}{article}{
  author={Shackleton, Kenneth J.},
  title={Combinatorial rigidity in curve complexes and mapping class groups},
  journal={Pacific J. Math.},
  volume={230},
  date={2007},
  number={1},
  pages={217--232},
  issn={0030-8730},
  review={\MR {2318453}},
  doi={10.2140/pjm.2007.230.217},
}

\bib{shinkle-arc}{article}{
  author={Shinkle, Emily},
  title={Finite rigid sets in arc complexes},
  journal={Algebr. Geom. Topol.},
  volume={20},
  date={2020},
  number={6},
  pages={3127--3145},
  issn={1472-2747},
  review={\MR {4185937}},
  doi={10.2140/agt.2020.20.3127},
}

\bib{shinkle-flip}{article}{
  author={Shinkle, Emily},
  title={Finite rigid sets in flip graphs},
  journal={Trans. Amer. Math. Soc.},
  volume={375},
  date={2021},
  number={02},
  pages={847--872},
  issn={0002-9947},
  review={\MR {4369237}},
  doi={10.1090/tran/8407},
}

 
\end{biblist}
\end{bibdiv}
\end{document}